\documentclass[a4paper, 11pt]{article}
\usepackage{fullpage,amsmath,amsthm,amssymb}

\newtheorem{theorem}{Theorem}
\newtheorem{prop}[theorem]{Proposition}
\newtheorem{lemma}[theorem]{Lemma}
\theoremstyle{definition}
\newtheorem{conj}{Conjecture}
\newcommand*{\abs}[1]{\lvert #1\rvert}
\newcommand*{\floor}[1]{\lfloor #1\rfloor}
\newcommand*{\ceil}[1]{\lceil #1\rceil}
\newcommand{\ie}{i.e.\ }

\title{Degrees in link graphs of regular graphs}
\author{Itai Benjamini and John Haslegrave}

\begin{document}
\maketitle
\begin{abstract}We analyse an extremal question on the degrees of the link graphs of a finite regular graph, that is, the subgraphs induced by non-trivial spheres.  We show that if $G$ is $d$-regular and connected but not complete then some link graph of $G$ has minimum degree at most $\floor{2d/3}-1$, and if $G$ is sufficiently large in terms of $d$ then some link graph has minimum degree at most $\floor{d/2}-1$; both bounds are best possible. We also give the corresponding best-possible result for the corresponding problem where subgraphs induced by balls, rather than spheres, are considered.

We motivate these questions by posing a conjecture concerning expansion of link graphs in large bounded-degree graphs, together with a heuristic justification thereof.\end{abstract}

\section{Link graphs with large degrees}\label{sec:sphere}

For a graph $G$, radius $r>0$ and vertex $v$, the sphere of radius $r$ about $v$, which we denote $\mathcal{S}_r(v)$, is the set of vertices at distance exactly $r$ from $v$, \ie $\mathcal{S}_r(v):=\{w\in V(G):d(v,w)=r\}$. If $\mathcal{S}_r(v)\neq\varnothing$, the $r$-link graph of $v$, which we denote $\mathcal{L}_r(v)$, is the induced subgraph $G[\mathcal{S}_r(v)]$. The graph $\mathcal{L}_1(e)$, where $G$ is a Cayley graph of a finitely-generated group, was used in \cite{Z} to give a sufficient condition for property (T).

Suppose that $G$ is a $d$-regular graph other than $K_{d+1}$. Can we choose $G$ in such a way as to ensure that each of the link graphs has large minimum degree? In other words, we wish to find the maximum value of $\min_{v,r}\delta(\mathcal{L}_r(v))$ over all non-complete $d$-regular graphs.

\begin{theorem}\label{thm:sphere}If $G$ is a non-complete connected $d$-regular graph then
\[\min_{v,r}\delta(\mathcal{L}_r(v))\leq\floor{2d/3}-1,\]
and this bound can be attained for any $d\equiv 2\pmod 3$.
\end{theorem}
\begin{proof}We verify the second statement first. Fix $d=3k-1$, so that $\floor{2d/3}-1=2k-2$, and let $G$ be the graph obtained by blowing up each vertex of $C_5$ to a clique of order $k$. Then $G$ is $(3k-1)$-regular and vertex-transitive with diameter $2$. For each $v$, the link graph $\mathcal{L}_1(v)$ consists of three cliques of order $k$, $k-1$ and $k$, with all vertices between the clique of order $k-1$ and the other two cliques, so has minimum degree $2k-2$. Furthermore, $\mathcal{L}_2(v)\cong K_{2k}$ and has minimum degree $2k-1$.

Next we prove the upper bound. Let $G$ be any non-complete $d$-regular graph, and note that this implies $G$ has no universal vertex. Suppose that every link graph of $G$ has minimum degree at least $m$. For any vertex $v$, any vertex in $\mathcal{S}_1(v)$ has at least $m$ neighbours within $\mathcal{S}_1(v)$ (and is a neighbour of $v$), so has at most $d-m-1$ neighbours in $\mathcal{S}_2(v)$. Choose $x\in\mathcal{S}_2(v)$ and $w\in\mathcal{S}_1(v)\cap\mathcal{S}_1(x)$. Now we must have $\deg_{\mathcal{L}_1(x)}(w)\geq m$. Since $w$ has at most $d-m-1$ neighbours in $\mathcal{S}_2(v)$, one of which is $x$, it has at most $d-m-2$ neighbours in $\mathcal{S}_2(v)\cap\mathcal{S}_1(x)$. Consequently, $w$ has at least $m-(d-m-2)=2m+2-d$ neighbours in $\mathcal{S}_1(v)\cap\mathcal{S}_1(x)$, and it follows that $x$ has at least $2m+3-d$ neighbours in $\mathcal{S}_1(v)$. However, considering $\deg_{\mathcal{L}_2(v)}(x)$, it also has at least $m$ neighbours in $\mathcal{S}_2(v)$. Since $\deg_G(x)=d$, we must have $3m+3-d\leq d$, \ie $m\leq 2d/3 -1$ and $m$ is an integer.
\end{proof}

Note that the example constructed to show that the bound of Theorem \ref{thm:sphere} is tight is small relative to $d$, and has diameter $2$. We can do better if we exclude such examples.

\begin{theorem}\label{thm:sphere-diam}For each $r\geq 3$, if $G$ is a connected $d$-regular graph with diameter at least $r$ then \[\min_{v,r}\delta(\mathcal{L}_r(v))\leq\floor{rd/(2r-1)}-1.\]\end{theorem}
\begin{proof}Suppose $G$ is such a graph and every link graph has minimum degree at least $m$. Fix a vertex $v$ such that $\mathcal{S}_r(v)\neq\varnothing$. We claim that for each $2\leq j\leq r$ every vertex in $\mathcal{S}_j(v)$ has at least $(2j-2)m-(j-1)d+2j-1$ neighbours in $\mathcal{S}_{j-1}(v)$. The result will follow from the case $j=r$ of the claim, since any vertex in $\mathcal{S}_r(v)$ also has at least $m$ neighbours in $\mathcal{S}_r(v)$, and so $(2j-1)m-(j-1)d+2j-1\leq d$, giving $m\leq jd/(2j-1)-1$.

We prove the claim by induction. The case $j=2$ was shown in the proof of Theorem \ref{thm:sphere}. For $j>2$, let $x$ be a vertex in $\mathcal{S}_j(v)$ and $w$ be a vertex in $\mathcal{S}_{j-1}(v)\cap\mathcal{S}_{1}(x)$. Since $w$ has at least $(2j-4)m-(j-2)d+2j-3$ neighbours in $\mathcal{S}_{j-2}(v)$ by the induction hypothesis, and at least $m$ neighbours in $\mathcal{S}_{j-1}(v)$, it has at most $(j-1)d-(2j-3)m-2j+2$ neighbours in $\mathcal{S}_j(v)\setminus\{x\}$. Since $w$ has at least $m$ neighbours in $\mathcal{S}_1(x)$, it must have at least $(2j-2)m-(j-1)d+2j-2$ neighbours in $\mathcal{S}_1(x)\cap\mathcal{S}_{j-1}(v)$, so $\abs{\mathcal{S}_1(x)\cap\mathcal{S}_{j-1}(v)}\geq (2j-2)m-(j-1)d+2j-1$, as required.
\end{proof}
From Theorem \ref{thm:sphere-diam} it follows that if $\operatorname{diam}(G)$ is sufficiently large in terms of $d$ (in particular, is at least $(d+1)/2$), and consequently whenever $\abs{G}$ is sufficiently large in terms of $d$, then \[\min_{v,r}\delta(\mathcal{L}_r(v))\leq\floor{d/2}-1.\]
In fact this bound is best possible.
\begin{prop}\label{cycle-power}For each even $d$ there exist connected $d$-regular graphs of arbitrarily large diameter satisfying $\min_{v,r}\delta(\mathcal{L}_r(v))=d/2-1$.\end{prop}
\begin{proof}Set $k=d/2$ and choose any $n$ satisfying $n-1\equiv a\pmod{2k}$ for some $a\in\{k,\ldots,2k\}$. Consider the $k$th power of the $n$-cycle, $C_n^k$.  This is $2k$-regular, and has diameter $s:=\ceil{(n-1)/(2k)}$, which may be made arbitrarily large. Every link graph of radius strictly smaller than $s$ consists of two cliques of order $k$ with possibly some edges in between, so has minimum degree at least $k-1$. A sphere of radius $s$ is a set of $b:=n-1-2k(s-1)$ consecutive vertices; since this value is congruent to $n-1$ modulo $2k$, positive and at most $2k$, we have $b=a\geq k$. Consequently $\mathcal{L}_s(v)$ will have minimum degree $k-1$.
\end{proof}

In this example, the diameter grows linearly with the order of the graph. It is natural to ask about large graphs which are well-connected, having diameter growing logarithmically with their order. In this case it is still possible for all link degrees to be linear in $d$. For example, start from a large cubic graph with logarithmic diameter and blow up each vertex to a clique of order $(d+1)/4$, where $d$ is fixed. The graph obtained is $d$-regular, and retains logarithmic diameter. Each link graph consists of some cliques, possibly with edges between them, and has minimum degree at least $(d-3)/4$. However, we do not know whether the constant $1/4$ in this example can be improved.

Additionally, in the example of Proposition \ref{cycle-power}, almost all links are $(d/2-1)$-regular. We might ask whether this is a necessary feature; in particular, do there exist $d$-regular graphs of large diameter for which all links have average degree greater than $d/2$? Note that in general there exist examples where all links have average degree significantly higher than the minimum degree over all links. For example, take the Cartesian product of a triangle and a long odd cycle, and blow up all vertices to cliques of order $(d+1)/5$. The average degree of each link is at least $(3d-7)/10$, but most links have minimum degree $(d-4)/5$. However, we do not know of a similar example where the minimum degree is close to $d/2$.

\section{Induced subgraphs on balls}\label{sec:ball}

In this section we consider a natural extension, replacing link graphs with the subgraphs induced by balls, \ie we consider induced subgraphs of the form $\overline{\mathcal{L}}_r(v):=G[\mathcal{B}_r(v)]$ where $\mathcal{B}_r(v):=\{w\in V(G):d(v,w)\leq r\}$. In this case we can easily have all minimum degrees of such graphs close to $d$ by taking $G$ to have only slightly more than $d+1$ vertices; for example, if $d$ is even then taking $G$ to be $K_{d+2}$ minus a perfect matching ensures $\min_{v,r}\delta(\overline{\mathcal{L}}_r(v))=d-2$.

However, again such graphs have diameter $2$, and if we require greater diameter we obtain non-trivial (and in fact tight) bounds.
\begin{theorem}\label{thm:ball}If $G$ is a connected $d$-regular graph with diameter at least $3$ then
\[\min_{v,r}\delta(\overline{\mathcal{L}}_r(v))\leq\floor{(2d-1)/3},\]
and this bound can be attained for any $d\equiv 2\pmod 3$ and arbitrarily large diameter.
\end{theorem}
\begin{proof}Fix $d=3k-1$, so that $\floor{2d/3}-1=2k-2$, choose $n\geq 6$ arbitrarily, and let $G$ be the graph obtained by blowing up each vertex of $C_n$ to a clique of order $k$. Then $G$ is $(3k-1)$-regular and vertex-transitive with diameter $\floor{n/2}\geq 3$. For each $v$ and $r$, the graph $\overline{\mathcal{L}}_r(v)$ consists of the cliques corresponding to $\min\{2r+1,n\}$ consecutive vertices of the cycle, and so has minimum degree at least $2k-1=(2d-1)/3$.

Suppose $G$ is $d$-regular with diameter at least $3$ and satisfies $\min_{v,r}\delta(\overline{\mathcal{L}}_r(v))=m$. Choose vertices $v,y$ with $d_G(v,y)=3$, and let $vwxy$ be a shortest path between them. Since $w$ has at least $m$ neighbours in $\mathcal{B}_1(v)$, it has at most $d-m$ neighbours in $\mathcal{S}_2(v)$. Since $w$ also has at least $m$ neighbours in $\mathcal{B}_1(x)$, it has at least $2m-d$ neighbours in $\mathcal{B}_1(x)\setminus\mathcal{S}_2(v)$, which must all be in $\mathcal{S}_1(v)$. It follows that $x$ has at least $2m-d+1$ neighbours in $\mathcal{S}_1(v)$. However, $x$ also has at least $m$ neighbours in $\mathcal{B}_1(y)$, which is disjoint from $\mathcal{S}_1(v)$. Thus $3m-d+1\leq\deg(x)\leq d$, and so $m\leq (2d-1)/3$.
\end{proof}

\section{A conjecture}

The results in Section \ref{sec:sphere} are in spirit indicating that spheres are not ``too connected''.
In this section we give a conjecture regarding $r$-links of infinite graphs in the same spirit.

A finite graph $G$ is said to have \textit{expansion} $h$ for
\[
h = \inf_{S \subset V_G :  \,\, 0 < \abs{S} \leq \abs{G}/2} \frac{\abs{\partial S}}{\abs{S}},
\]
where $V_{G}$ are the vertices of $G$ and $\partial S$ is the outer vertex boundary of $S$.

An \textit{expander family} is a sequence of graphs such that, for some $h > 0$, all graphs in the sequence have expansion at least $h$.

\begin{conj}\label{conjecture}There is no sequence of bounded-degree finite graphs, with size growing to infinity,
such that all links in all the graphs form an expander family.
\end{conj}
Note that size growing to infinity is equivalent to diameter growing to infinity. In the extremal examples discussed in Section \ref{sec:sphere} (other than the example with diameter $2$), the link graphs do not expand, since they are disconnected (and for the example with logarithmic diameter, typically have many components).

In what follows, we provide some heuristic support for Conjecture \ref{conjecture}. Recall that an infinite graph $G$ is said to be \textit{amenable}
if 
\[
\inf_{S \subset V_G :  \,\, 0 < \abs{S} <\infty} \frac{\abs{\partial S}}{\abs{S}}=0.
\]
\begin{lemma}\label{lemma1}
Assume $G$ is an infinite graph, and all the $r$-links of $G$ form an expander family.
Then $G$ is non-amenable.
\end{lemma}
\begin{proof}
First we show that, for any vertex $v$, the size of $\mathcal{S}_r(v)$ grows to infinity with $r$. Indeed, if there is some $a$ for which $\abs{\mathcal{S}_r(v)}\leq a$ infinitely often, there is in particular some $r>a$ with this property, and $\mathcal{S}_r(v)$ is a cutset separating  $\mathcal{B}_{r-1}(v)$ from infinity. Let $w$ be a vertex which lies on an infinite ray proceeding from $v$, at distance $s>r$. Now there are $r>a$ distances in the set $\{s-r+1,\ldots,s\}$, and so at least one of $\mathcal{S}_{s-r+1}(w),\ldots,\mathcal{S}_s(w)$ must fail to intersect $\mathcal{S}_r(v)$. Since each of these spheres intersects both $\mathcal{B}_{r-1}(v)$ and its complement, at least one is disconnected and so not an expander.

Let $S$ be a finite set of vertices in $G$. Pick a vertex $v$ in $G$ far enough from $S$,
such that the $r$-links around $v$ that intersect $S$ have size larger than $2^{\abs{S}}$.
By uniform expansion of the $r$-links, in each $r$-link that intersects $S$, the intersection of $S$ with the $r$-link has proportional boundary, and
as the radius of the spheres varies they give a disjoint cover of $S$. Thus every finite set of vertices in $G$ has proportional boundary and thus $G$
is non-amenable.\end{proof}

The heuristic for Conjecture \ref{conjecture} follows the strategy of Salez \cite{S}.
By Lemma~\ref{lemma1}, the BS limit of graphs in which all links are expanders with expansion bounded away from zero is a.s.\ non-amenable.
A simple random walk on a non-amenable graph has positive linear speed.
The BS limit is a unimodular random graph (see \cite{AL, BS} for definitions)
in which all the $r$-links of the root are an expander family.
In \cite{BK} it was proved that if you add edges so that the vertices of each of the
levels of a binary tree (\ie each of the links of the root) form a sequence of uniform expanders, then the resulting graph has no non-constant bounded harmonic functions.
We \emph{believe}, (but cannot prove) that the same holds for unimodular random graphs, \ie if the links of the root of a unimodular random graph form an expander family then it has no non-constant bounded harmonic functions.
Since this is equivalent to having zero speed \cite{BC}, it would lead to a contradiction.

\medskip
Our result in Section \ref{sec:ball} equally relates to a similar question. In \cite{B} it was asked: Is there is a sequence of finite
bounded-degree graphs growing in size to infinity, so that all the induced
balls in all the graphs in the sequence form an expander family?
For related results on heat kernel supports see \cite{FV}.

\section*{Acknowledgements}
JH was supported by the UK Research and Innovation Future Leaders Fellowship number
MR/S016325/1.

\end{document}